\documentclass[reqno]{amsart}
\usepackage{amssymb}
\usepackage{amsfonts}

\newtheorem{theorem}{Theorem}
\theoremstyle{plain}

\newtheorem{claim}{Claim}

\newtheorem{conjecture}{Conjecture}
\newtheorem{corollary}{Corollary}

\newtheorem{definition}{Definition}

\newtheorem{lemma}{Lemma}

\newtheorem{proposition}{Proposition}
\newtheorem{remark}{Remark}

\numberwithin{equation}{section}
\numberwithin{theorem}{section}
\numberwithin{algorithm}{section}
\numberwithin{axiom}{section}
\numberwithin{case}{section}
\numberwithin{claim}{section}
\numberwithin{conclusion}{section}
\numberwithin{condition}{section}
\numberwithin{conjecture}{section}
\numberwithin{corollary}{section}
\numberwithin{criterion}{section}
\numberwithin{definition}{section}
\numberwithin{example}{section}
\numberwithin{exercise}{section}
\numberwithin{lemma}{section}
\numberwithin{notation}{section}
\numberwithin{problem}{section}
\numberwithin{proposition}{section}
\numberwithin{remark}{section}
\numberwithin{solution}{section}

\input{tcilatex}

\begin{document}
\title[A perturbation approach for Paneitz energy]{A perturbation approach
for Paneitz energy on standard three sphere}
\date{}
\author{Fengbo Hang}
\address{Courant Institute, New York University, 251 Mercer Street, New York
NY 10012}
\email{fengbo@cims.nyu.edu}
\author{Paul C. Yang}
\address{Department of Mathematics, Princeton University, Fine Hall,
Washington Road, Princeton NJ 08544}
\email{yang@math.princeton.edu}

\begin{abstract}
We present another proof of the sharp inequality for Paneitz operator on the
standard three sphere, in the spirit of subcritical approximation for the
classical Yamabe problem. To solve the perturbed problem, we use a
symmetrization process which only works for extremal functions. This gives a
new example of symmetrization for higher order variational problems.
\end{abstract}

\maketitle

\section{Introduction\label{sec1}}

Fourth order Paneitz operator (\cite{Br, P}) gains interest due to its role
in the progress of four dimensional conformal geometry (see \cite{CGY, HY5}%
). On a three dimensional Riemannian manifold $\left( M,g\right) $, it is
given by%
\begin{equation}
P\varphi =\Delta ^{2}\varphi +4\func{div}\left( Rc\left( \nabla \varphi
,e_{i}\right) e_{i}\right) -\frac{5}{4}\func{div}\left( R\nabla \varphi
\right) -\frac{1}{2}Q\varphi .  \label{eq1.1}
\end{equation}%
Here%
\begin{equation}
Q=-\frac{1}{4}\Delta R-2\left\vert Rc\right\vert ^{2}+\frac{23}{32}R^{2},
\label{eq1.2}
\end{equation}%
$Rc$ is the Ricci curvature, $R$ is the scalar curvature and $%
e_{1},e_{2},e_{3}$ is a local orthonormal frame. The Paneitz operator
satisfies the transformational law%
\begin{equation}
P_{\rho ^{-4}g}\varphi =\rho ^{7}P_{g}\left( \rho \varphi \right)
\label{eq1.3}
\end{equation}%
for any smooth positive function $\rho $.

This is similar to the conformal Laplacian operator $L\varphi =-\frac{%
4\left( n-1\right) }{n-2}\Delta \varphi +R\varphi $ on a Riemannian manifold
with dimension $n\geq 3$, which satisfies%
\begin{equation}
L_{\rho ^{\frac{4}{n-2}}g}\varphi =\rho ^{-\frac{n+2}{n-2}}L_{g}\left( \rho
\varphi \right)  \label{eq1.4}
\end{equation}%
for any positive smooth function $\rho $.

On the standard sphere $S^{n}$,%
\begin{equation}
L=-\frac{4\left( n-1\right) }{n-2}\Delta +n\left( n-1\right)  \label{eq1.5}
\end{equation}%
is positive definite. Its associated quadratic form%
\begin{equation}
\int_{S^{n}}L\varphi \cdot \varphi d\mu =\int_{S^{n}}\left[ \frac{4\left(
n-1\right) }{n-2}\left\vert \nabla \varphi \right\vert ^{2}+n\left(
n-1\right) \varphi ^{2}\right] d\mu  \label{eq1.6}
\end{equation}%
satisfies the sharp Sobolev inequality (\cite[Chapter 5]{SY})%
\begin{equation}
\inf_{\varphi \in H^{1}\left( S^{n}\right) \backslash \left\{ 0\right\} }%
\frac{\int_{S^{n}}\left[ \frac{4\left( n-1\right) }{n-2}\left\vert \nabla
\varphi \right\vert ^{2}+n\left( n-1\right) \varphi ^{2}\right] d\mu }{%
\left\Vert \varphi \right\Vert _{L^{\frac{2n}{n-2}}}^{2}}=n\left( n-1\right)
\left\vert S^{n}\right\vert ^{\frac{2}{n}}.  \label{eq1.7}
\end{equation}%
Here $\left\vert S^{n}\right\vert $ is the volume of $S^{n}$. In another
word the the functional minimizes at $\varphi =1$.

On standard three sphere $S^{3}$,%
\begin{equation}
P\varphi =\Delta ^{2}\varphi +\frac{1}{2}\Delta \varphi -\frac{15}{16}\varphi
\label{eq1.8}
\end{equation}%
and it has a negative first eigenvalue $-\frac{15}{16}$. Denote%
\begin{eqnarray}
E\left( \varphi \right) &=&\int_{S^{3}}P\varphi \cdot \varphi d\mu
\label{eq1.9} \\
&=&\int_{S^{3}}\left[ \left( \Delta \varphi \right) ^{2}-\frac{1}{2}%
\left\vert \nabla \varphi \right\vert ^{2}-\frac{15}{16}\varphi ^{2}\right]
d\mu ,  \notag
\end{eqnarray}%
and we look at the Paneitz energy (see \cite[section 1]{HY1}) $\left\Vert
\varphi ^{-1}\right\Vert _{L^{6}}^{2}E\left( \varphi \right) $ for $\varphi
\in H^{2}\left( S^{3}\right) $ with $\varphi >0$. Note the condition $%
\varphi >0$ makes sense since $H^{2}\left( S^{3}\right) \subset C^{\frac{1}{2%
}}\left( S^{3}\right) $ by the Sobolev embedding theorem. It is not clear at
all whether Paneitz energy even has a lower bound. Nevertheless it is proven
in \cite{YZ} that

\begin{theorem}[\protect\cite{YZ}]
\label{thm1.1}For any $\varphi \in H^{2}\left( S^{3}\right) $ with $\varphi
>0$, we have%
\begin{equation}
\left\Vert \varphi ^{-1}\right\Vert _{L^{6}}^{2}\int_{S^{3}}\left[ \left(
\Delta \varphi \right) ^{2}-\frac{1}{2}\left\vert \nabla \varphi \right\vert
^{2}-\frac{15}{16}\varphi ^{2}\right] d\mu \geq -\frac{15}{16}\left\vert
S^{3}\right\vert ^{\frac{4}{3}}.  \label{eq1.10}
\end{equation}%
In another word, the Paneitz energy minimizes at $\varphi =1$.
\end{theorem}

All minimizers can be identified by the Liouville type theorem in \cite{CX}.

To continue, we observe that a classical way to derive (\ref{eq1.7}) is by
the subcritical approximation. If we try to prove the extremal problem (\ref%
{eq1.7}) indeed has a minimizer, the direct method in calculus of variations
meets the problem that the embedding $H^{1}\left( S^{n}\right) \subset L^{%
\frac{2n}{n-2}}\left( S^{n}\right) $ is not compact. On the other hand, for $%
2<q<\frac{2n}{n-2}$, we have a compact embedding $H^{1}\left( S^{n}\right)
\subset L^{q}\left( S^{n}\right) $, this together with standard argument
shows we can find a positive smooth function $u$ such that%
\begin{equation}
\inf_{\varphi \in H^{1}\left( S^{n}\right) \backslash \left\{ 0\right\} }%
\frac{\int_{S^{n}}\left[ \frac{4\left( n-1\right) }{n-2}\left\vert \nabla
\varphi \right\vert ^{2}+n\left( n-1\right) \varphi ^{2}\right] d\mu }{%
\left\Vert \varphi \right\Vert _{L^{q}}^{2}}  \label{eq1.11}
\end{equation}%
is achieved at $u$. Moreover after scaling, $u$ satisfies%
\begin{equation}
Lu=u^{q-1}\quad \text{on }S^{n}.  \label{eq1.12}
\end{equation}%
The method of moving plane (\cite{GNN}) tells us (\ref{eq1.12}) has only
constant solution. Hence%
\begin{equation}
\inf_{\varphi \in H^{1}\left( S^{n}\right) \backslash \left\{ 0\right\} }%
\frac{\int_{S^{n}}\left[ \frac{4\left( n-1\right) }{n-2}\left\vert \nabla
\varphi \right\vert ^{2}+n\left( n-1\right) \varphi ^{2}\right] d\mu }{%
\left\Vert \varphi \right\Vert _{L^{q}}^{2}}=n\left( n-1\right) \left\vert
S^{n}\right\vert ^{1-\frac{2}{q}}.  \label{eq1.13}
\end{equation}%
Let $q\uparrow \frac{2n}{n-2}$, we get (\ref{eq1.7}). We also note that this
subcritical approximation process can be used in the solution to Yamabe
problem (see \cite[Chapter 5]{SY}).

Now let us go back to (\ref{eq1.10}). There have been several different ways
to derive this inequality in \cite{H, HY1,YZ}, however none of them is in
the spirit of the subcritical approximation approach above for (\ref{eq1.7}%
). The main aim of this note is to provide such a method. As we will see
soon, the solution to the approximation problem provides a new way of doing
symmetrization for variational problems associated with higher order
equations.

To motivate the formulation of the perturbation problem, we recall some
observations in \cite{HY1,XY}. Let us consider the extremal problem%
\begin{eqnarray}
&&\inf_{\varphi \in H^{2}\left( M\right) ,\varphi >0}\left\Vert \varphi
^{-1}\right\Vert _{L^{6}\left( M\right) }^{2}E\left( \varphi \right) ;
\label{eq1.14} \\
E\left( \varphi \right) &=&\int_{M}P\varphi \cdot \varphi d\mu  \notag \\
&=&\int_{M}\left[ \left( \Delta \varphi \right) ^{2}-4Rc\left( \nabla
\varphi ,\nabla \varphi \right) +\frac{5}{4}R\left\vert \nabla \varphi
\right\vert ^{2}-\frac{1}{2}Q\varphi ^{2}\right] d\mu  \notag
\end{eqnarray}%
on a smooth compact three dimensional Riemannian manifold $\left( M,g\right) 
$. We say the Paneitz operator satisfies condition P (see \cite[Section 5]%
{HY1}) if for any $\varphi \in H^{2}\left( M\right) \backslash \left\{
0\right\} $, $\varphi \left( p\right) =0$ for some $p$ would imply $E\left(
\varphi \right) >0$. Under condition P, the extremal problem (\ref{eq1.14})
is achieved and the set of minimizers is compact under $C^{\infty }$
topology. Due to the noncompactness of Mobius transformation group on the
standard $S^{3}$, the Paneitz operator can not satisfy condition P in this
singular case. Nevertheless it satisfies the condition NN (see \cite[%
Sections 5 and 7]{HY1}): for any $\varphi \in H^{2}\left( M\right)
\backslash \left\{ 0\right\} $, $\varphi \left( p\right) =0$ for some $p$
would imply $E\left( \varphi \right) \geq 0$. The Green's function of $P$
can be written as (see \cite{HY1, HY6})%
\begin{equation}
G_{P}\left( x,y\right) =-\frac{\left\Vert x-y\right\Vert }{8\pi }
\label{eq1.15}
\end{equation}%
for $x,y\in S^{3}$, the unit sphere in $\mathbb{R}^{4}$. In particular for
any $x\in S^{3}$, $E\left( G_{P}\left( x,\cdot \right) \right) =0$. In fact,
the constant multiple of Green's function are the only function $\varphi \in
H^{2}\left( S^{3}\right) $ satisfying $\varphi \left( p\right) =0$ for some $%
p$ and $E\left( \varphi \right) =0$. To pass from condition NN to condition
P, we replace the Paneitz operator $P$ by $P+\varepsilon $ for $\varepsilon
>0$ small. Denote%
\begin{equation}
E_{\varepsilon }\left( \varphi \right) =\int_{S^{3}}\left( P\varphi
+\varepsilon \varphi \right) \varphi d\mu =E\left( \varphi \right)
+\varepsilon \left\Vert \varphi \right\Vert _{L^{2}}^{2}.  \label{eq1.16}
\end{equation}%
Then we study the variational problem%
\begin{equation}
\inf_{\varphi \in H^{2}\left( S^{3}\right) ,\varphi >0}\left\Vert \varphi
^{-1}\right\Vert _{L^{6}}^{2}E_{\varepsilon }\left( \varphi \right) .
\label{eq1.17}
\end{equation}%
It is straightforward to show the extremal problem has a minimizer (see
Lemma \ref{lem2.1}). If $u$ is a minimizer, then $u$ must be positive smooth
functions and after scaling it satisfies%
\begin{equation}
Pu+\varepsilon u=-u^{-7}\quad \text{on }S^{3}.  \label{eq1.18}
\end{equation}%
If we can prove the Liouville type theorem that every solution to (\ref%
{eq1.18}) must be constant function, then we know for any $\varphi \in
H^{2}\left( S^{3}\right) $, $\varphi >0$ and $\varepsilon >0$ small,%
\begin{equation}
\left\Vert \varphi ^{-1}\right\Vert _{L^{6}}^{2}\left( E\left( \varphi
\right) +\varepsilon \left\Vert \varphi \right\Vert _{L^{2}}^{2}\right) \geq
\left\vert S^{3}\right\vert ^{\frac{1}{3}}\left( E\left( 1\right)
+\varepsilon \left\vert S^{3}\right\vert \right) .  \label{eq1.19}
\end{equation}%
Letting $\varepsilon \downarrow 0$, we get (\ref{eq1.10}). But unfortunately
we can not verify such a Liouville theorem at this stage. There are many
evidence such a statement should be correct. We put it as a conjecture.

\begin{conjecture}
Let $P=\Delta ^{2}+\frac{1}{2}\Delta -\frac{15}{16}$ be the Paneitz operator
on the standard three sphere, $\varepsilon >0$ be a small positive number.
If $u$ is a positive smooth function on $S^{3}$ such that $Pu+\varepsilon
u=-u^{-7}$, then $u$ must be a constant function.
\end{conjecture}

Without such a Liouville type result, we need another approach to show $1$
is a minimizer to the perturbed variational problem (\ref{eq1.17}). This
consists the main part of the note. Our approach is motivated from \cite%
{BT,HY4,HWY,R}. At first we would like to show that every minimizer must be
radial symmetric and decreasing with respect to some point on $S^{3}$. This
will be achieved by a symmetrization process which only works for minimizers
(in contrast, usual symmetrization process works for all test functions (see 
\cite{PS})). Next it follows from the usual Kazdan-Warner type identity that
any minimizer must be a constant function. It is worth pointing out the
method of symmetrization usually does not work well for higher order
variational problems or higher order Sobolev spaces. The point of our
approach lies in that it provides a new way to overcome the difficulty.

\section{The perturbation problem\label{sec2}}

For $\varepsilon >0$ small, we denote%
\begin{equation}
-s_{\varepsilon }=\inf_{\varphi \in H^{2}\left( S^{3}\right) ,\varphi
>0}\left\Vert \varphi ^{-1}\right\Vert _{L^{6}}^{2}E_{\varepsilon }\left(
\varphi \right) .  \label{eq2.1}
\end{equation}%
Here $E_{\varepsilon }\left( \varphi \right) $ is defined in (\ref{eq1.16}).
Since $E_{\varepsilon }\left( 1\right) <0$, by our choice of notation we see 
$s_{\varepsilon }>0$ and%
\begin{equation}
s_{\varepsilon }=\sup_{\varphi \in H^{2}\left( S^{3}\right) ,\varphi
>0}\left\Vert \varphi ^{-1}\right\Vert _{L^{6}}^{2}\left( -E_{\varepsilon
}\left( \varphi \right) \right) .  \label{eq2.2}
\end{equation}

\begin{lemma}
\label{lem2.1}For $\varepsilon >0$ small, there exists a $u\in C^{\infty
}\left( S^{3}\right) $ such that $u>0$ and $\left\Vert u^{-1}\right\Vert
_{L^{6}}^{2}E_{\varepsilon }\left( u\right) =-s_{\varepsilon }$. By scaling
we can assume $\left\Vert u^{-1}\right\Vert _{L^{6}}=1$, then%
\begin{equation}
Pu+\varepsilon u=-s_{\varepsilon }u^{-7}.  \label{eq2.3}
\end{equation}
\end{lemma}

\begin{proof}
We can find a sequence $u_{i}\in H^{2}\left( S^{3}\right) $ such that $%
u_{i}>0$ and $\left\Vert u_{i}^{-1}\right\Vert _{L^{6}}^{2}E_{\varepsilon
}\left( u_{i}\right) \rightarrow -s_{\varepsilon }$. By scaling we can
assume $\max_{S^{3}}u_{i}=1$. Since $E_{\varepsilon }\left( u_{i}\right) <0$%
, we see $\left\Vert u_{i}\right\Vert _{H^{2}\left( S^{3}\right) }\leq c$.
Hence after passing to a subsequence there exists a $u\in H^{2}\left(
S^{3}\right) $ with $u_{i}\rightharpoonup u$ weakly in $H^{2}\left(
S^{3}\right) $. It follows that $u_{i}\rightarrow u$ uniformly on $S^{3}$
and hence $\max_{S^{3}}u=1$. Note that $E_{\varepsilon }\left( u\right) \leq
0$, hence $u$ must be strictly positive everywhere. Indeed if $u\left(
p\right) =0$ for some $p$, it follows from the fact $P$ satisfies condition
NN (see \cite[Sections 5 and 7]{HY1}) that $E\left( u\right) \geq 0$, hence $%
E_{\varepsilon }\left( u\right) \geq \varepsilon \left\Vert u\right\Vert
_{L^{2}}^{2}>0$, a contradiction.

Since $u$ is positive everywhere, we see $u_{i}^{-1}\rightarrow u^{-1}$
uniformly on $S^{3}$. Hence%
\begin{equation*}
-s_{\varepsilon }\leq \left\Vert u^{-1}\right\Vert
_{L^{6}}^{2}E_{\varepsilon }\left( u\right) \leq \lim \inf_{i\rightarrow
\infty }\left\Vert u_{i}^{-1}\right\Vert _{L^{6}}^{2}E_{\varepsilon }\left(
u_{i}\right) =-s_{\varepsilon }.
\end{equation*}%
In another word, $u$ is a minimizer.

After scaling we can assume $\left\Vert u^{-1}\right\Vert _{L^{6}}=1$, it is
clear that $u$ satisfies (\ref{eq2.3}) weakly. Standard elliptic theory
tells us $u$ must be smooth.
\end{proof}

\section{Basics for symmetrization on the sphere\label{sec3}}

We discuss symmetrization on $S^{n}$ for all $n$ since there is no
difference between different $n$'s. Let $S^{n}\subset \mathbb{R}^{n+1}$ be
the unit sphere. For $x,y\in \mathbb{R}^{n+1}$, $x\cdot y$ denotes the usual
dot product. Let $N=\left( 0,\cdots ,0,1\right) $ be the north pole and $%
S=\left( 0,\cdots ,0,-1\right) $ be the south pole. We write $\overline{%
\mathbb{R}}=\mathbb{R}\cup \left\{ \pm \infty \right\} $.

\begin{definition}
\label{def3.1}Fix a unit vector $u\in S^{n}$. We say a function $%
f:S^{n}\rightarrow \overline{\mathbb{R}}$ is radial symmetric and decreasing
with respect to $u$ if there exists a function $\phi :\left[ -1,1\right]
\rightarrow \overline{\mathbb{R}}$ such that $\phi $ is increasing and $%
f\left( x\right) =\phi \left( x\cdot u\right) $ for all $x\in S^{n}$.
Similarly we have the notion of radial symmetric and increasing with respect
to $u$.
\end{definition}

Let $f:S^{n}\rightarrow \overline{\mathbb{R}}$ be a measurable function,
then its symmetrization $f^{\ast }:S^{n}\rightarrow \overline{\mathbb{R}}$
is the unique function satisfying the following conditions:

\begin{itemize}
\item $f^{\ast }$ is radial symmetric and decreasing with respect to the
north pole $N$ and for any $t\in \mathbb{R}$,%
\begin{equation*}
\left\vert f>t\right\vert =\left\vert f^{\ast }>t\right\vert .
\end{equation*}%
Here for any set $E\subset S^{n}$, $\left\vert E\right\vert $ denotes the
measure of $E$.

\item $\left. f^{\ast }\right\vert _{S^{n}\backslash \left\{ S\right\} }$ is
lower semicontinuous and $f^{\ast }\left( S\right) =\limfunc{essinf}%
_{S^{n}}f $. Here $S$ is the south pole.
\end{itemize}

We also recall some notations from \cite[section 3]{HY5}. If $K=K\left(
x,y\right) $ is a function on $S^{n}\times S^{n}$, then we say $K$ is a
kernel (function) on $S^{n}$. If $K$ is a nice kernel on $S^{n}$ and $f$ is
a nice function on $S^{n}$, then we write%
\begin{equation}
\left( T_{K}f\right) \left( x\right) =\int_{S^{n}}K\left( x,y\right) f\left(
y\right) d\mu \left( y\right) .  \label{eq3.1}
\end{equation}%
If $K^{\prime }$ is another kernel on $S^{n}$, then we define their
convolution as%
\begin{equation}
\left( K\ast K^{\prime }\right) \left( x,z\right) =\int_{S^{n}}K\left(
x,y\right) K^{\prime }\left( y,z\right) d\mu \left( y\right) .  \label{eq3.2}
\end{equation}

A special class of kernel is pointed out in \cite[section 2]{BT}: if $k$ is
a function on the interval $\left[ -1,1\right] $, then we have a kernel $%
K_{k}$ on $S^{n}$ given by%
\begin{equation}
K_{k}\left( x,y\right) =k\left( x\cdot y\right) .  \label{eq3.3}
\end{equation}%
The orthogornal group $\limfunc{O}\left( n+1\right) $ acts naturally on $%
S^{n}\times S^{n}$ as $R\cdot \left( x,y\right) =\left( Rx,Ry\right) $ for
any $R\in \limfunc{O}\left( n+1\right) $, $x,y\in S^{n}$. The above class of
kernels are exactly those kernels invariant under the group action. So we
call this kind of kernel as invariant kernel. It is easy to see the
convolution of invariant kernels are still invariant. If $k_{1}$ and $k_{2}$
are nice functions on $\left[ -1,1\right] $, then we denote the function on $%
\left[ -1,1\right] $ associated with $K_{k_{1}}\ast K_{k_{2}}$ as $k_{1}\ast
k_{2}$. In another word, for any $t\in \left[ -1,1\right] $,%
\begin{equation}
\left( k_{1}\ast k_{2}\right) \left( t\right) =\int_{S^{n}}k_{1}\left(
x\cdot y\right) k_{2}\left( y\cdot z\right) d\mu \left( y\right)
\label{eq3.4}
\end{equation}%
for any $x,z\in S^{n}$ with $x\cdot z=t$. Note that $k_{1}\ast k_{2}$
depends on $n$, we will use $k_{1}\ast _{n}k_{2}$ when confusion could
happen.

\begin{lemma}
\label{lem3.1} If $k_{1}$ and $k_{2}$ are both increasing functions on $%
\left[ -1,1\right] $, then so is $k_{1}\ast k_{2}$.
\end{lemma}

\begin{proof}
Without losing of generality, we can assume both $k_{1}$ and $k_{2}$ are
smooth functions. Fix two perpendicular unit vectors $u$ and $v$. For $%
0<\theta <\pi $, we denote%
\begin{equation}
p\left( \theta \right) =u\cos \theta +v\sin \theta ,  \label{eq3.5}
\end{equation}%
then%
\begin{equation}
p^{\prime }\left( \theta \right) =-u\sin \theta +v\cos \theta .
\label{eq3.6}
\end{equation}%
By definition of $k_{1}\ast k_{2}$,%
\begin{equation}
\left( k_{1}\ast k_{2}\right) \left( \cos \theta \right)
=\int_{S^{n}}k_{1}\left( u\cdot x\right) k_{2}\left( x\cdot p\left( \theta
\right) \right) d\mathcal{H}^{n}\left( x\right) .  \label{eq3.7}
\end{equation}%
Here $\mathcal{H}^{n}$ denotes the Hausdorff measure. We use this notation
instead of $\mu $ to make the application of coarea formula easier.
Differentiate with respect to $\theta $, we get%
\begin{eqnarray}
&&-\sin \theta \left( k_{1}\ast k_{2}\right) ^{\prime }\left( \cos \theta
\right)  \label{eq3.8} \\
&=&\int_{S^{n}}k_{1}\left( u\cdot x\right) k_{2}^{\prime }\left( x\cdot
p\right) \left( x\cdot p^{\prime }\right) d\mathcal{H}^{n}\left( x\right) 
\notag \\
&=&\int_{-1}^{1}d\tau \frac{k_{2}^{\prime }\left( \tau \right) }{\sqrt{%
1-\tau ^{2}}}\int_{\substack{ x\in S^{n}  \\ x\cdot p=\tau }}k_{1}\left(
u\cdot x\right) \left( x\cdot p^{\prime }\right) d\mathcal{H}^{n-1}\left(
x\right) .  \notag
\end{eqnarray}%
We have applied the coarea formula in the second equality. Now Lemma \ref%
{lem3.1} follows from the following claim.

\begin{claim}
\label{claim3.1}%
\begin{equation}
\int_{\substack{ x\in S^{n}  \\ x\cdot p=\tau }}k_{1}\left( u\cdot x\right)
\left( x\cdot p^{\prime }\right) d\mathcal{H}^{n-1}\left( x\right) \leq 0.
\label{eq3.9}
\end{equation}
\end{claim}

Indeed, we have the reflection%
\begin{equation}
y=x-2\left( x\cdot p^{\prime }\right) p^{\prime }.  \label{eq3.10}
\end{equation}%
Note that%
\begin{eqnarray}
&&\int_{\substack{ x\in S^{n}  \\ x\cdot p=\tau }}k_{1}\left( u\cdot
x\right) \left( x\cdot p^{\prime }\right) d\mathcal{H}^{n-1}\left( x\right)
\label{eq3.11} \\
&=&\int_{\substack{ x\in S^{n}  \\ x\cdot p=\tau  \\ x\cdot p^{\prime }>0}}%
k_{1}\left( u\cdot x\right) \left( x\cdot p^{\prime }\right) d\mathcal{H}%
^{n-1}\left( x\right) +\int_{\substack{ x\in S^{n}  \\ x\cdot p=\tau  \\ %
x\cdot p^{\prime }<0}}k_{1}\left( u\cdot x\right) \left( x\cdot p^{\prime
}\right) d\mathcal{H}^{n-1}\left( x\right)  \notag \\
&=&\int_{\substack{ x\in S^{n}  \\ x\cdot p=\tau  \\ x\cdot p^{\prime }>0}}%
k_{1}\left( u\cdot x\right) \left( x\cdot p^{\prime }\right) d\mathcal{H}%
^{n-1}\left( x\right)  \notag \\
&&-\int_{\substack{ y\in S^{n}  \\ y\cdot p=\tau  \\ y\cdot p^{\prime }<0}}%
k_{1}\left( u\cdot \left( y-2\left( y\cdot p^{\prime }\right) p^{\prime
}\right) \right) \left( y\cdot p^{\prime }\right) d\mathcal{H}^{n-1}\left(
y\right)  \notag \\
&=&\int_{\substack{ x\in S^{n}  \\ x\cdot p=\tau  \\ x\cdot p^{\prime }>0}}%
\left( k_{1}\left( u\cdot x\right) -k_{1}\left( u\cdot \left( x-2\left(
x\cdot p^{\prime }\right) p^{\prime }\right) \right) \right) \left( x\cdot
p^{\prime }\right) d\mathcal{H}^{n-1}\left( x\right)  \notag \\
&\leq &0  \notag
\end{eqnarray}%
because $u\cdot p^{\prime }=-\sin \theta \leq 0$ and $k_{1}$ is increasing.
\end{proof}

For a function $k$ on $\left[ -1,1\right] $, we denote%
\begin{equation}
T_{k}=T_{K_{k}}.  \label{eq3.12}
\end{equation}%
In another word,%
\begin{equation}
\left( T_{k}f\right) \left( x\right) =\int_{S^{n}}k\left( x\cdot y\right)
f\left( y\right) d\mu \left( y\right) .  \label{eq3.13}
\end{equation}

\begin{itemize}
\item $T_{k}1=const$ i.e. $1$ is an eigenfunction of $T_{k}$.

\item If we denote $\lambda =T_{k}1$, then for any nice function $f$ on $%
S^{n}$,%
\begin{equation}
\int_{S^{n}}T_{k}f\left( x\right) d\mu \left( x\right) =\lambda
\int_{S^{n}}fd\mu .  \label{eq3.14}
\end{equation}

\item For $R\in \limfunc{O}\left( n+1\right) $, we let%
\begin{equation}
f_{R}\left( x\right) =f\left( Rx\right)  \label{eq3.15}
\end{equation}%
for all $x\in S^{n}$. Then%
\begin{equation}
T_{k}\left( f_{R}\right) =\left( T_{k}f\right) _{R}.  \label{eq3.16}
\end{equation}
\end{itemize}

As a consequence, we see if $u\in S^{n}$ and $f$ is radial symmetric with
respect to $u$, then so is $T_{k}f$.

\begin{lemma}
\label{lem3.2}Assume $k$ is increasing function on $\left[ -1,1\right] $, $%
u\in S^{n}$, $f:S^{n}\rightarrow \mathbb{R}$ is symmetric and decreasing
with respect to $u$, then so is $T_{k}f$.
\end{lemma}

\begin{proof}
We assume $f\left( x\right) =\phi \left( x\cdot u\right) $. Then there
exists a function $\psi $ such that $T_{k}f\left( x\right) =\psi \left(
x\cdot u\right) $. Without losing of generality we can assume $k$ and $f$
are both smooth functions. Choose a unit vector $v$ such that $u\cdot v=0$,
let $p\left( \theta \right) =u\cos \theta +v\sin \theta $, then%
\begin{eqnarray*}
\psi \left( \cos \theta \right) &=&T_{k}f\left( p\left( \theta \right)
\right) \\
&=&\int_{S^{n}}k\left( p\cdot x\right) \phi \left( x\cdot u\right) d\mathcal{%
H}^{n}\left( x\right) .
\end{eqnarray*}%
The argument in the proof of Lemma \ref{lem3.1} tells us $\psi $ is
increasing. Lemma \ref{lem3.2} follows.
\end{proof}

Later on we will need an analogue of the classical Riesz rearrangement
inequality proved in \cite{BT}.

\begin{proposition}[Theorem 2 in \protect\cite{BT}]
\label{prop3.1}Let $k$ be a bounded increasing function on $\left[ -1,1%
\right] $, then for any $f,g\in L^{1}\left( S^{n}\right) $, we have%
\begin{eqnarray}
&&\int_{S^{n}}\int_{S^{n}}k\left( x\cdot y\right) f\left( x\right) g\left(
y\right) d\mu \left( x\right) d\mu \left( y\right)  \label{eq3.17} \\
&\leq &\int_{S^{n}}\int_{S^{n}}k\left( x\cdot y\right) f^{\ast }\left(
x\right) g^{\ast }\left( y\right) d\mu \left( x\right) d\mu \left( y\right) .
\notag
\end{eqnarray}
\end{proposition}

Here we are going to discuss some results on when the equality is achieved.
We will not aim for most general statements, \ but restrict to case
sufficient for our purpose.

\begin{lemma}
\label{lem3.3}If $k$ is a bounded strictly increasing function on $\left[
-1,1\right] $ and $f,g\in C\left( S^{n},\mathbb{R}\right) $ such that%
\begin{eqnarray}
&&\int_{S^{n}}\int_{S^{n}}k\left( x\cdot y\right) f\left( x\right) g\left(
y\right) d\mu \left( x\right) d\mu \left( y\right)  \label{eq3.18} \\
&=&\int_{S^{n}}\int_{S^{n}}k\left( x\cdot y\right) f^{\ast }\left( x\right)
g^{\ast }\left( y\right) d\mu \left( x\right) d\mu \left( y\right)  \notag
\end{eqnarray}%
and both $f$ and $g$ are not constant functions, then $f$ and $g$ must be
radial symmetric and decreasing with respect to the same point.
\end{lemma}

To continue, we recall some basic operations in \cite{BT}. Let $u$ be a unit
vector, then we have a half space%
\begin{equation*}
H=\left\{ x\in \mathbb{R}^{n+1}:x\cdot u\geq 0\right\} .
\end{equation*}%
Note that $u$ is simply the inner normal direction at $\partial H$. Let $R$
be the reflection with respect to $\partial H$. For any function $f$ on $%
S^{n}$, we denote%
\begin{equation*}
f_{H}\left( x\right) =\left\{ 
\begin{array}{cc}
\max \left\{ f\left( x\right) ,f\left( Rx\right) \right\} , & \text{if }x\in
H\cap S^{n}; \\ 
\min \left\{ f\left( x\right) ,f\left( Rx\right) \right\} , & \text{if }x\in
S^{n}\backslash H.%
\end{array}%
\right.
\end{equation*}%
Note that $f_{H}$ has the same distribution functions as $f$. In particular, 
$\left( f_{H}\right) ^{\ast }=f^{\ast }$.

\begin{lemma}
\label{lem3.4}If $f\in C\left( S^{n},\mathbb{R}\right) $ such that for every
half space $H$ and $x,y\in S^{n}\cap H$ we have%
\begin{equation*}
\left( f\left( x\right) -f\left( Rx\right) \right) \left( f\left( y\right)
-f\left( Ry\right) \right) \geq 0,
\end{equation*}%
here $R$ is the reflection with respect to $\partial H$, then $f$ must be
radial symmetric and decreasing with respect to some point.
\end{lemma}

\begin{proof}
At first we note that if $x\in S^{n}\cap H$ such that $f\left( x\right)
>f\left( Rx\right) $, then for every $y\in S^{n}\cap H$, $f\left( y\right)
\geq f\left( Ry\right) $. We will use this assertion repeatedly.

The lemma will be proved by induction on $n$. First assume $n=1$. In this
case we identify $\mathbb{R}^{2}$ with $\mathbb{C}$ as usual. We can assume $%
f$ is not a constant function, then we fix a number $t$ such that $%
\min_{S^{1}}f<t<\max_{S^{1}}f$. Let $E=\left\{ f<t\right\} $ and the
connected components of $E$ are open arcs. After rotation, we can assume one
connected component is written as $\left\{ e^{i\theta }:-\delta <\theta
<\delta \right\} $, here $0<\delta <\pi $. For any $0<\alpha <\pi $, we
claim $f\left( e^{i\alpha }\right) =f\left( e^{-i\alpha }\right) $, that is, 
$f$ is symmetric with respect to the real axis. Indeed, for $0<\varepsilon
<\min \left\{ \alpha ,\delta \right\} $, let $H$ be the top half plane cut
by line in the direction $e^{i\frac{\varepsilon }{2}}$. Since $f\left(
e^{i\delta }\right) =t>f\left( e^{i\left( -\delta +\varepsilon \right)
}\right) $, we see $f\left( e^{i\alpha }\right) \geq f\left( e^{i\left(
-\alpha +\varepsilon \right) }\right) $. Let $\varepsilon \downarrow 0$, we
see $f\left( e^{i\alpha }\right) \geq f\left( e^{-i\alpha }\right) $. Repeat
the same argument with $e^{-i\delta }$ in place with $e^{i\delta }$, we see $%
f\left( e^{-i\alpha }\right) \geq f\left( e^{i\alpha }\right) $. The claim
follows. Next we want to show if $0\leq \beta <\alpha \leq \pi $, then $%
f\left( e^{i\alpha }\right) \geq f\left( e^{i\beta }\right) $ if $\alpha
-\beta <\delta $. As a consequence, we know $f$ is radial symmetric and
decreasing with respect to $-1$. Indeed let $H$ be the top half plane cut by
the line in the direction $e^{i\frac{\alpha -\beta }{2}}$. Since $f\left(
e^{i\delta }\right) >f\left( e^{i\left( -\delta +\alpha -\beta \right)
}\right) $, we see $f\left( e^{i\alpha }\right) \geq f\left( e^{-i\beta
}\right) =f\left( e^{i\beta }\right) $. The claim follows.

Assume for some $n\geq 2$, the lemma is correct for functions on $S^{n-1}$,
we want to show it is also valid on $S^{n}$. Again we assume $f$ is not a
constant function and $\min_{S^{n}}f<t<\max_{S^{n}}f$. Let $E=\left\{
f<t\right\} $. For $p\in S^{n}$ and $0<r\leq \pi $, we denote $%
B_{r}^{S^{n}}\left( p\right) $ as the open geodesic ball of radius $r$ on $%
S^{n}$ with center $p$. We can find a geodesic ball with largest radius
among all such balls contained in $E$. By rotation we assume it is given by $%
B_{\delta }^{S^{n}}\left( S\right) $, here $S$ is the south pole. Then there
two different points $p_{1},p_{2}$ in $\partial B_{\delta }^{S^{n}}\left(
N\right) \cap \partial E$. Note that $f\left( p_{1}\right) =f\left(
p_{2}\right) =t$. Let $L$ be any hyperplane (which may not pass the origin)
in $\mathbb{R}^{n+1}$ containing $p_{1}$ and $p_{2}$, then by applying the
induction hyphothesis on $\left. f\right\vert _{L\cap S^{n}}$, we conclude
that $\left. f\right\vert _{S^{n}\backslash B_{\delta }^{S^{n}}\left(
S\right) }\geq t$. Hence $\left. f\right\vert _{\partial B_{\delta
}^{S^{n}}\left( S\right) }=t$.

Now we claim $f$ is radial symmetric and decreasing with respect to $N$.

Indeed assume $0<\alpha <\pi $, $u^{\prime },u^{\prime \prime }\in S^{n-1}$
and $u^{\prime }\neq u^{\prime \prime }$, we will show%
\begin{equation*}
f\left( u^{\prime }\sin \alpha ,\cos \alpha \right) =f\left( u^{\prime
\prime }\sin \alpha ,\cos \alpha \right) .
\end{equation*}%
For $\alpha <\beta <\pi $ and $\beta $ close to $\alpha $, we will show%
\begin{equation*}
f\left( u^{\prime }\sin \alpha ,\cos \alpha \right) \geq f\left( u^{\prime
\prime }\sin \beta ,\cos \beta \right) .
\end{equation*}%
Let $\beta \downarrow \alpha $, we get%
\begin{equation*}
f\left( u^{\prime }\sin \alpha ,\cos \alpha \right) \geq f\left( u^{\prime
\prime }\sin \alpha ,\cos \alpha \right) .
\end{equation*}%
Switching $u^{\prime }$ and $u^{\prime \prime }$, we get the inequality in
the other direction and hence the equality. Let%
\begin{equation*}
u=\frac{\left( u^{\prime }\sin \alpha ,\cos \alpha \right) -\left( u^{\prime
\prime }\sin \beta ,\cos \beta \right) }{\left\Vert \left( u^{\prime }\sin
\alpha ,\cos \alpha \right) -\left( u^{\prime \prime }\sin \beta ,\cos \beta
\right) \right\Vert }.
\end{equation*}%
Since $\beta $ is close to $\alpha $, $u$ is close to $\left( \frac{%
u^{\prime }-u^{\prime \prime }}{\left\Vert u^{\prime }-u^{\prime \prime
}\right\Vert },0\right) $. Let $H=\left\{ x\in \mathbb{R}^{n+1}:x\cdot u\geq
0\right\} $, and $R$ be the reflection with respect to $\partial H$. Let%
\begin{equation*}
x=\left( \frac{u^{\prime }-u^{\prime \prime }}{\left\Vert u^{\prime
}-u^{\prime \prime }\right\Vert }\sin \delta ,-\cos \delta \right) ,
\end{equation*}%
note that $Rx=x-2\left( x\cdot u\right) u$. Since $x\cdot u$ is close to $%
\sin \delta $, we know $x\cdot u>0$. Hence 
\begin{equation*}
\left( Rx\right) _{n+1}=-\cos \delta -2\left( x\cdot u\right) u_{n+1}<-\cos
\delta
\end{equation*}%
and it follows that $Rx\in B_{\delta }^{S^{n}}\left( S\right) $. In
particular $f\left( x\right) >f\left( Rx\right) $. Since%
\begin{equation*}
\left( u^{\prime }\sin \alpha ,\cos \alpha \right) \cdot u>0
\end{equation*}%
and%
\begin{equation*}
R\left( u^{\prime }\sin \alpha ,\cos \alpha \right) =\left( u^{\prime \prime
}\sin \beta ,\cos \beta \right) ,
\end{equation*}%
we see%
\begin{equation*}
f\left( u^{\prime }\sin \alpha ,\cos \alpha \right) \geq f\left( u^{\prime
\prime }\sin \beta ,\cos \beta \right) .
\end{equation*}%
It follows that $f$ is radial symmetric and decreasing with respect to $N$.
\end{proof}

Next we generalize the Lemma \ref{lem3.4} to two functions setting.

\begin{lemma}
\label{lem3.5}If $f,g\in C\left( S^{n},\mathbb{R}\right) $ such that for
every half space $H$ and $x,y\in S^{n}\cap H$ we have%
\begin{equation*}
\left( f\left( x\right) -f\left( Rx\right) \right) \left( g\left( y\right)
-g\left( Ry\right) \right) \geq 0,
\end{equation*}%
here $R$ is the reflection with respect to $\partial H$, and both $f$ and $g$
are not constant functions, then $f$ and $g$ must be radial symmetric and
decreasing with respect to the same point.
\end{lemma}

To reduce Lemma \ref{lem3.5} to Lemma \ref{lem3.4}, we need some basic facts
about generators of the orthogornal group $\limfunc{O}\left( n+1\right) $.
Let $u\in S^{n}$, we have the hyperplane%
\begin{equation*}
u^{\perp }=\left\{ x\in \mathbb{R}^{n+1}:x\cdot u=0\right\} .
\end{equation*}%
The reflection with respect to $u^{\perp }$ is given by%
\begin{equation*}
R_{u}x=x-2\left( x\cdot u\right) u.
\end{equation*}

\begin{lemma}
\label{lem3.6}For any $u_{0}\in S^{n}$ and $\varepsilon >0$, the set%
\begin{equation*}
\left\{ R_{u}:u\in S^{n},\left\Vert u-u_{0}\right\Vert <\varepsilon \right\}
\end{equation*}%
generates $\limfunc{O}\left( n+1\right) $.
\end{lemma}

\begin{proof}
We prove by induction on $n$. For convenience we denote $G$ as the subgroup
generated by the set $\left\{ R_{u}:u\in S^{n},\left\Vert u-u_{0}\right\Vert
<\varepsilon \right\} $. We hope to show $G=\limfunc{O}\left( n+1\right) $.

Assume $n=1$. As usual we identify $\mathbb{R}^{2}$ with $\mathbb{C}$. We
can write $u_{0}=e^{i\alpha _{0}}$, then for any $0<\delta <\varepsilon $,
we have%
\begin{equation*}
\left\vert e^{i\left( \alpha _{0}+\delta \right) }-e^{i\alpha
_{0}}\right\vert <\delta <\varepsilon
\end{equation*}%
and%
\begin{equation*}
R_{e^{i\left( \alpha _{0}+\delta \right) }}R_{e^{i\alpha _{0}}}e^{i\theta
}=e^{i\left( 2\delta +\theta \right) }.
\end{equation*}%
Hence all rotation lies in $G$. This together with the fact $R_{u_{0}}\in G$
implies $G=\limfunc{O}\left( 2\right) $.

Assume for some $n\geq 2$, the lemma is true for $\limfunc{O}\left( n\right) 
$, we will show it is also true for $\limfunc{O}\left( n+1\right) $. Give $%
v\in S^{n}$ satisfying $v\perp u_{0}$, if $O\in \limfunc{O}\left( n+1\right) 
$ such that $Ov=v$, then $\left. O\right\vert _{v^{\perp }}:v^{\perp
}\rightarrow v^{\perp }$ is an orthogonal transformation too. By induction
hypothesis we know $O\in G$. Next for any $u\in S^{n}$, using $n\geq 2$, we
can find $v\in S^{n}$ such that $v\perp u$ and $v\perp u_{0}$. It follows
that $R_{u}v=v$ and hence $R_{u}\in G$. Since every orthogonal
transformation is a finite composition of reflections, we see $G=\limfunc{O}%
\left( n+1\right) $.
\end{proof}

Now we proceed to prove Lemma \ref{lem3.5}.

\begin{proof}[Proof of Lemma \protect\ref{lem3.5}]
We first claim for every half space $H$ and $x,y\in S^{n}\cap H$ we have%
\begin{equation*}
\left( f\left( x\right) -f\left( Rx\right) \right) \left( f\left( y\right)
-f\left( Ry\right) \right) \geq 0.
\end{equation*}%
Indeed if this is not the case, then for some $H$ and $x_{0},y_{0}\in H\cap
S^{n}$ such that%
\begin{equation*}
\left( f\left( x_{0}\right) -f\left( Rx_{0}\right) \right) \left( f\left(
y_{0}\right) -f\left( Ry_{0}\right) \right) <0.
\end{equation*}%
We may assume $f\left( x_{0}\right) >f\left( Rx_{0}\right) $ and $f\left(
y_{0}\right) <f\left( Ry_{0}\right) $. Then for any $x\in H\cap S^{n}$,%
\begin{eqnarray*}
\left( f\left( x_{0}\right) -f\left( Rx_{0}\right) \right) \left( g\left(
x\right) -g\left( Rx\right) \right) &\geq &0; \\
\left( f\left( y_{0}\right) -f\left( Ry_{0}\right) \right) \left( g\left(
x\right) -g\left( Rx\right) \right) &\geq &0.
\end{eqnarray*}%
It follows that $g\left( x\right) =g\left( Rx\right) $. Hence $g=g_{R}$. Let 
$u_{0}$ be the inner normal direction of $H$, then $x_{0}\cdot u_{0}>0$, $%
y_{0}\cdot u_{0}>0$. For $\varepsilon >0$ small and $\left\Vert
u-u_{0}\right\Vert <\varepsilon $, we still have $x_{0}\cdot u>0$, $%
y_{0}\cdot u>0$, $f\left( x_{0}\right) >f\left( R_{u}x_{0}\right) $ and $%
f\left( y_{0}\right) <f\left( R_{u}y_{0}\right) $. It follows that $%
g=g_{R_{u}}$. By Lemma \ref{lem3.6} all these $R_{u}$'s generate $\limfunc{O}%
\left( n+1\right) $, hence $g=g_{O}$ for any $O\in \limfunc{O}\left(
n+1\right) $. This implies $g$ must be a constant function and it
contradicts with the assumption.

By Lemma \ref{lem3.4} we know $f$ must be radial symmetric and decreasing
with respect to some point on $S^{n}$. Without losing of generality we
assume that point in the north pole $N$. Since $f$ is not a constant
function, we can find $t$ such that $f\left( S\right) <t<f\left( N\right) $.
Then%
\begin{equation*}
\left\{ f<t\right\} =B_{\delta }^{S^{n}}\left( S\right)
\end{equation*}%
for some $\delta \in \left( 0,\pi \right) $. The argument at the end of
proof of Lemma \ref{lem3.4} tells us $g$ must be radial symmetric and
decreasing with respect to $N$ too.
\end{proof}

We are ready to prove Lemma \ref{lem3.3}.

\begin{proof}[Proof of Lemma \protect\ref{lem3.3}]
We prove it by a contradiction argument. If the conclusion is not true, then
by Lemma \ref{lem3.5} we can find a half space $H$ and $x_{0},y_{0}\in H\cap
S^{n}$ such that%
\begin{equation*}
\left( f\left( x_{0}\right) -f\left( Rx_{0}\right) \right) \left( g\left(
y_{0}\right) -g\left( Ry_{0}\right) \right) <0.
\end{equation*}%
Here $R$ denotes the reflection with respect to $\partial H$. Without losing
of generality we assume $f\left( x_{0}\right) >f\left( Rx_{0}\right) $ and $%
g\left( y_{0}\right) <g\left( Ry_{0}\right) $. Note that $x_{0},y_{0}\notin
\partial H$. By continuity of $f$ and $g$ we can find a neighborhood of $%
x_{0}$ in $S^{n}$, namely $U\left( x_{0}\right) $, such that $U\left(
x_{0}\right) \subset H$ and for every $x\in U\left( x_{0}\right) $, $f\left(
x\right) >f\left( Rx\right) $. Similarly, we find $U\left( y_{0}\right)
\subset H$ such that for every $y\in U\left( y_{0}\right) $, $g\left(
y\right) <g\left( Ry\right) $. For function $F$ and $G$ on $S^{n}$, we have%
\begin{eqnarray*}
&&\int_{S^{n}}\int_{S^{n}}k\left( x\cdot y\right) F\left( x\right) G\left(
y\right) d\mu \left( x\right) d\mu \left( y\right) \\
&=&\int_{H\cap S^{n}}\int_{H\cap S^{n}}\left[ k\left( x\cdot y\right) \left(
F\left( x\right) G\left( y\right) +F\left( Rx\right) G\left( Ry\right)
\right) \right. \\
&&\left. +k\left( Rx\cdot y\right) \left( F\left( Rx\right) G\left( y\right)
+F\left( x\right) G\left( Ry\right) \right) \right] d\mu \left( x\right)
d\mu \left( y\right) .
\end{eqnarray*}%
A careful but elementary calculation shows (see \cite[proof of lemma 1]{BT})
for $x,y\in H\cap S^{n}$,%
\begin{eqnarray*}
&&k\left( x\cdot y\right) \left( f\left( x\right) g\left( y\right) +f\left(
Rx\right) g\left( Ry\right) \right) +k\left( Rx\cdot y\right) \left( f\left(
Rx\right) g\left( y\right) +f\left( x\right) g\left( Ry\right) \right) \\
&\leq &k\left( x\cdot y\right) \left( f_{H}\left( x\right) g_{H}\left(
y\right) +f_{H}\left( Rx\right) g_{H}\left( Ry\right) \right) \\
&&+k\left( Rx\cdot y\right) \left( f_{H}\left( Rx\right) g_{H}\left(
y\right) +f_{H}\left( x\right) g_{H}\left( Ry\right) \right) .
\end{eqnarray*}%
For $x\in U\left( x_{0}\right) $ and $y\in U\left( y_{0}\right) $, \ this
inequality is strict. Indeed in this case,%
\begin{eqnarray*}
&&\limfunc{RHS}-\limfunc{LHS} \\
&=&\left( k\left( x\cdot y\right) -k\left( Rx\cdot y\right) \right) \left(
f\left( x\right) -f\left( Rx\right) \right) \left( g\left( Ry\right)
-g\left( y\right) \right) \\
&>&0.
\end{eqnarray*}%
Here $\limfunc{RHS}$ and $\limfunc{LHS}$ mean the right hand side and left
hand side respectively of the above inequality. It follows that%
\begin{eqnarray*}
&&\int_{S^{n}}\int_{S^{n}}k\left( x\cdot y\right) f\left( x\right) g\left(
y\right) d\mu \left( x\right) d\mu \left( y\right) \\
&<&\int_{S^{n}}\int_{S^{n}}k\left( x\cdot y\right) f_{H}\left( x\right)
g_{H}\left( y\right) d\mu \left( x\right) d\mu \left( y\right) \\
&\leq &\int_{S^{n}}\int_{S^{n}}k\left( x\cdot y\right) \left( f_{H}\right)
^{\ast }\left( x\right) \left( g_{H}\right) ^{\ast }\left( y\right) d\mu
\left( x\right) d\mu \left( y\right) \\
&=&\int_{S^{n}}\int_{S^{n}}k\left( x\cdot y\right) f^{\ast }\left( x\right)
g^{\ast }\left( y\right) d\mu \left( x\right) d\mu \left( y\right) .
\end{eqnarray*}%
This gives us a contradiction with the equality (\ref{eq3.18}).
\end{proof}

\section{Every minimizer of perturbation problem must be radial symmetric
and decreasing\label{sec4}}

Recall that for any $O\in \limfunc{O}\left( 4\right) $ and function $f$ on $%
S^{3}$, we write $f_{O}\left( x\right) =f\left( Ox\right) $ for $x\in S^{3}$%
. The Paneitz operator $P$ is invariant under the orthogornal group i.e. for
any $f\in C^{\infty }\left( S^{3}\right) $ and $O\in \limfunc{O}\left(
4\right) $, $\left( Pf\right) _{O}=Pf_{O}$. Hence the Green's function $%
G_{P} $ is invariant too. For convenience we denote%
\begin{equation}
K=-G_{P}.  \label{eq4.1}
\end{equation}%
Then%
\begin{equation}
K\left( x,y\right) =\frac{\left\Vert x-y\right\Vert }{8\pi }\quad \text{for }%
x,y\in S^{3}.  \label{eq4.2}
\end{equation}%
The associated single variable function $k$ is given by%
\begin{equation}
k\left( t\right) =\frac{\sqrt{1-t}}{4\sqrt{2}\pi }\quad \text{for }t\in %
\left[ -1,1\right] .  \label{eq4.3}
\end{equation}%
Similarly for $\varepsilon >0$ small, $G_{P+\varepsilon }$ is also
invariant. Let $K_{\varepsilon }=-G_{P+\varepsilon }$ and the associated
single variable function is denoted as $k_{\varepsilon }$. Since%
\begin{equation}
\left( P+\varepsilon \right) ^{-1}=\sum_{j=0}^{\infty }\left( -\varepsilon
\right) ^{j}\left( P^{-1}\right) ^{j+1},  \label{eq4.4}
\end{equation}%
we see%
\begin{equation}
G_{P+\varepsilon }=\sum_{j=0}^{\infty }\left( -\varepsilon \right) ^{j}%
\overset{j+1}{\overbrace{G_{P}\ast \cdots \ast G_{P}}}.  \label{eq4.5}
\end{equation}%
It follows that%
\begin{equation}
K_{\varepsilon }=\sum_{j=0}^{\infty }\varepsilon ^{j}\overset{j+1}{%
\overbrace{K\ast \cdots \ast K}}  \label{eq4.6}
\end{equation}%
and%
\begin{equation}
k_{\varepsilon }=\sum_{j=0}^{\infty }\varepsilon ^{j}\overset{j+1}{%
\overbrace{k\ast \cdots \ast k}}.  \label{eq4.7}
\end{equation}

\begin{lemma}
\label{lem4.1} For $\varepsilon >0$ small, $k_{\varepsilon }$ is strictly
decreasing on $\left[ -1,1\right] $.
\end{lemma}

\begin{proof}
We have%
\begin{equation}
k_{\varepsilon }=k+\varepsilon k\ast k+\sum_{j=1}^{\infty }\varepsilon ^{2j}%
\overset{2j}{\overbrace{k\ast \cdots \ast k}}\ast \left( k+\varepsilon k\ast
k\right) .  \label{eq4.8}
\end{equation}%
By Lemma \ref{lem3.1} we only need to know $k+\varepsilon k\ast k$ is
strictly decreasing. By the formula of $k$ in (\ref{eq4.3}) we only need to
show $k\ast k\in C^{1}\left( \left[ -1,1\right] \right) $. As in the proof
of Lemma \ref{lem3.1} we fix two perpendicular unit vectors $u$ and $v$ in $%
\mathbb{R}^{4}$. For $0<\theta <\pi $,%
\begin{equation}
p\left( \theta \right) =u\cos \theta +v\sin \theta ,  \label{eq4.9}
\end{equation}%
then%
\begin{equation}
p^{\prime }\left( \theta \right) =-u\sin \theta +v\cos \theta .
\label{eq4.10}
\end{equation}%
We have%
\begin{equation}
\left( k\ast k\right) \left( \cos \theta \right) =\int_{S^{3}}k\left( u\cdot
x\right) k\left( x\cdot p\left( \theta \right) \right) d\mathcal{H}%
^{3}\left( x\right) .  \label{eq4.11}
\end{equation}%
Here $k$ is given in (\ref{eq4.3}) and%
\begin{equation}
k^{\prime }\left( t\right) =-\frac{1}{8\sqrt{2}\pi }\frac{1}{\sqrt{1-t}}%
\quad \text{for }-1\leq t<1.  \label{eq4.12}
\end{equation}%
Hence for $0<\theta <\pi $,%
\begin{equation}
-\sin \theta \left( k\ast k\right) ^{\prime }\left( \cos \theta \right)
=\int_{S^{3}}k\left( u\cdot x\right) k^{\prime }\left( x\cdot p\right) \cdot
\left( x\cdot p^{\prime }\right) d\mathcal{H}^{3}\left( x\right) .
\label{eq4.13}
\end{equation}%
For convenience we denote%
\begin{equation}
\phi \left( \theta \right) =\int_{S^{3}}k\left( u\cdot x\right) k^{\prime
}\left( x\cdot p\right) \cdot \left( x\cdot p^{\prime }\right) d\mathcal{H}%
^{3}\left( x\right)  \label{eq4.14}
\end{equation}%
for $0\leq \theta \leq \pi $. It is clear that $\phi \in C\left( \left[
0,\pi \right] \right) $. We need to show $\frac{\phi \left( \theta \right) }{%
\sin \theta }$ has a limit at both $0$ and $\pi $. Let $O$ be an orthogornal
matrix with $Ou=p$ and $Ov=p^{\prime }$. Then%
\begin{equation}
O^{-1}u=u\cos \theta -v\sin \theta .  \label{eq4.15}
\end{equation}%
We make a change of variable $x=Oy$ in (\ref{eq4.14}) and get%
\begin{equation}
\phi \left( \theta \right) =\int_{S^{3}}k\left( y\cdot \left( u\cos \theta
-v\sin \theta \right) \right) k^{\prime }\left( y\cdot u\right) \cdot \left(
y\cdot v\right) d\mathcal{H}^{3}\left( y\right) .  \label{eq4.16}
\end{equation}%
It follows from this formula that $\phi \in C^{1}\left( \left[ 0,\pi \right]
\right) $. Moreover%
\begin{eqnarray}
\phi \left( 0\right) &=&\int_{S^{3}}k\left( x\cdot u\right) k^{\prime
}\left( x\cdot u\right) \cdot \left( x\cdot v\right) d\mathcal{H}^{3}\left(
x\right) =0,  \label{eq4.17} \\
\phi \left( \pi \right) &=&\int_{S^{3}}k\left( -x\cdot u\right) k^{\prime
}\left( x\cdot u\right) \cdot \left( x\cdot v\right) d\mathcal{H}^{3}\left(
x\right) =0,  \label{eq4.18}
\end{eqnarray}%
hence $\frac{\phi }{\sin \theta }\in C\left( \left[ 0,\pi \right] \right) $.
Lemma \ref{lem4.1} follows.
\end{proof}

\begin{remark}
\label{rmk4.1}Here is another way to prove Lemma \ref{lem4.1}. First we
observe that $K_{\varepsilon }\left( x,y\right) -K\left( x,y\right) $ is $%
C^{2}$ is variable $\left( \varepsilon ,x,y\right) $. We have%
\begin{equation}
\psi \left( \theta \right) =K\left( \left( \sin \theta ,0,0,\cos \theta
\right) ,\left( 0,0,0,1\right) \right) =k\left( \cos \theta \right)
\label{eq4.19}
\end{equation}%
and%
\begin{equation}
\psi _{\varepsilon }\left( \theta \right) =K_{\varepsilon }\left( \left(
\sin \theta ,0,0,\cos \theta \right) ,\left( 0,0,0,1\right) \right)
=k_{\varepsilon }\left( \cos \theta \right) .  \label{eq4.20}
\end{equation}%
Note that $\psi _{\varepsilon }\left( \theta \right) -\psi \left( \theta
\right) $ in $C^{2}$ in $\left( \varepsilon ,\theta \right) $. Since for $%
t\geq 0$ small, $\psi _{\varepsilon }\left( \pi -t\right) =\psi
_{\varepsilon }\left( \pi +t\right) $, we see $\psi _{\varepsilon }^{\prime
}\left( \pi \right) =0$. On the other hand, for $0\leq \theta \leq \pi $,%
\begin{equation}
\psi \left( \theta \right) =\frac{1}{4\pi }\sin \frac{\theta }{2}.
\label{eq4.21}
\end{equation}%
We have%
\begin{equation}
\psi ^{\prime }\left( \theta \right) =\frac{1}{8\pi }\cos \frac{\theta }{2}>0
\label{eq4.22}
\end{equation}%
for $0\leq \theta <\pi $ and%
\begin{equation}
\psi ^{\prime \prime }\left( \theta \right) =-\frac{1}{16\pi }\sin \frac{%
\theta }{2}<0  \label{eq4.23}
\end{equation}
for $0<\theta \leq \pi $. These together with the fact $\psi _{\varepsilon
}^{\prime }\left( \pi \right) =0$ implies $\psi _{\varepsilon }^{\prime
}\left( \theta \right) >0$ for $0\leq \theta <\pi $ if $\varepsilon $ is
small enough. Hence $\psi _{\varepsilon }$ is strictly increasing on $\left[
0,\pi \right] $ and $k_{\varepsilon }$ is strictly decreasing.
\end{remark}

Let $u$ be a minimizer for (\ref{eq2.1}) with $\left\Vert u^{-1}\right\Vert
_{L^{6}}=1$, then%
\begin{equation}
Pu+\varepsilon u=-s_{\varepsilon }u^{-7}.  \label{eq4.24}
\end{equation}%
We want to show $u$ must be radial symmetric and decreasing with respect to
some point on $S^{3}$. Let $v$ be the smooth function on $S^{3}$ solving%
\begin{equation}
Pv+\varepsilon v=-s_{\varepsilon }\left( u^{\ast }\right) ^{-7}.
\label{eq4.25}
\end{equation}%
Then%
\begin{equation}
v=s_{\varepsilon }T_{k_{\varepsilon }}\left( \left( u^{\ast }\right)
^{-7}\right) .  \label{eq4.26}
\end{equation}%
Denote%
\begin{equation}
\alpha =-E_{\varepsilon }\left( v\right) ,  \label{eq4.27}
\end{equation}%
then%
\begin{eqnarray}
\alpha &=&-\int_{S^{3}}\left( P+\varepsilon \right) v\cdot vd\mu
\label{eq4.28} \\
&=&s_{\varepsilon }^{2}\int_{S^{3}}T_{k_{\varepsilon }}\left( \left( u^{\ast
}\right) ^{-7}\right) \cdot \left( u^{\ast }\right) ^{-7}d\mu  \notag \\
&=&s_{\varepsilon }^{2}\int_{S^{3}}T_{k_{\varepsilon }}\left( \left( \left(
u^{-7}\right) ^{\ast }\right) _{R}\right) \cdot \left( \left( u^{-7}\right)
^{\ast }\right) _{R}d\mu  \notag \\
&=&s_{\varepsilon }^{2}\int_{S^{3}}\left( T_{k_{\varepsilon }}\left( \left(
u^{-7}\right) ^{\ast }\right) \right) _{R}\cdot \left( \left( u^{-7}\right)
^{\ast }\right) _{R}d\mu  \notag \\
&=&s_{\varepsilon }^{2}\int_{S^{3}}T_{k_{\varepsilon }}\left( \left(
u^{-7}\right) ^{\ast }\right) \cdot \left( u^{-7}\right) ^{\ast }d\mu  \notag
\\
&\leq &s_{\varepsilon }^{2}\int_{S^{3}}T_{k_{\varepsilon }}\left(
u^{-7}\right) \cdot u^{-7}d\mu  \notag \\
&=&s_{\varepsilon }\int_{S^{3}}u^{-6}d\mu  \notag \\
&=&s_{\varepsilon }.  \notag
\end{eqnarray}%
Here $R$ is the reflection given by%
\begin{equation}
R\left( x_{1},x_{2},x_{3},x_{4}\right) =\left(
x_{1},x_{2},x_{3},-x_{4}\right) \quad \text{for }x\in S^{3}.  \label{eq4.29}
\end{equation}%
We have used Proposition \ref{prop3.1} in between.

On the other hand, we have%
\begin{equation}
\alpha =s_{\varepsilon }\int_{S^{3}}\left( u^{\ast }\right) ^{-7}vd\mu ,
\label{eq4.30}
\end{equation}%
hence%
\begin{eqnarray}
\alpha \left\Vert v^{-1}\right\Vert _{L^{6}} &=&s_{\varepsilon }\left\Vert
\left( u^{\ast }\right) ^{-7}v\right\Vert _{L^{1}}\left\Vert
v^{-1}\right\Vert _{L^{6}}  \label{eq4.31} \\
&\geq &s_{\varepsilon }\left\Vert \left( u^{\ast }\right) ^{-7}\right\Vert
_{L^{\frac{6}{7}}}  \notag \\
&=&s_{\varepsilon }\left\Vert \left( u^{\ast }\right) ^{-1}\right\Vert
_{L^{6}}^{7}  \notag \\
&=&s_{\varepsilon }\left\Vert u^{-1}\right\Vert _{L^{6}}^{7}  \notag \\
&=&s_{\varepsilon }.  \notag
\end{eqnarray}%
This together with the fact $\alpha \leq s_{\varepsilon }$ tells us $%
\left\Vert v^{-1}\right\Vert _{L^{6}}\geq 1$. We have%
\begin{equation}
s_{\varepsilon }\geq -E_{\varepsilon }\left( v\right) \left\Vert
v^{-1}\right\Vert _{L^{6}}^{2}=\alpha \left\Vert v^{-1}\right\Vert
_{L^{6}}\cdot \left\Vert v^{-1}\right\Vert _{L^{6}}\geq s_{\varepsilon }.
\label{eq4.32}
\end{equation}%
Hence $\left\Vert v^{-1}\right\Vert _{L^{6}}=1$ and $\alpha =s_{\varepsilon
} $. It follows from Lemma \ref{lem3.3} and (\ref{eq4.28}) that $u^{-7}$
must be radial and decreasing with respect to some point on $S^{3}$ and so
is $u$.

\section{A new proof of the sharp Sobolev inequality\label{sec5}}

Following \cite{HWY}, we will first derive a Kazdan-Warner type condition
and then use it to show the minimizer $u$ must be a constant function.

As in \cite[section 3]{HWY}, we first introduce the weighted total $Q$
curvature functional. If $\left( M^{3},g\right) $ is a smooth compact
Riemannian manifold, and $\chi $ is a positive smooth function on $M$, we
define%
\begin{equation}
I\left( M,g,\chi \right) =\left( \int_{M}\chi d\mu \right) ^{\frac{1}{3}%
}\int_{M}Qd\mu .  \label{eq5.1}
\end{equation}%
For $\widetilde{g}\in \left[ g\right] $, we write $\widetilde{g}=\rho ^{-4}g$%
, then%
\begin{equation}
I\left( M,\widetilde{g},\chi \right) =-2\left( \int_{M}\chi \rho ^{-6}d\mu
\right) ^{\frac{1}{3}}\int_{M}P\rho \cdot \rho d\mu .  \label{eq5.2}
\end{equation}%
The Euler-Lagrange equation of this functional reads as%
\begin{equation}
P\rho =\text{const}\cdot \chi \rho ^{-7}.  \label{eq5.3}
\end{equation}

\begin{lemma}[Kazdan-Warner type condition]
\label{lem5.1}Assume $\left( M^{3},g\right) $ is a smooth compact Riemannian
manifold, and $\chi $ and $\rho $ are positive smooth functions on $M$
satisfying%
\begin{equation}
P\rho =-\chi \rho ^{-7}.  \label{eq5.4}
\end{equation}%
Let $X$ be a conformal vector field on $\left( M,g\right) $, then%
\begin{equation}
\int_{M}X\chi \cdot \rho ^{-6}d\mu =0.  \label{eq5.5}
\end{equation}
\end{lemma}

\begin{proof}
The proof is exactly the same as for \cite[Lemma 3.1]{HWY}. Indeed let $\phi
_{t}$ be the 1-parameter group generated by $X$, then%
\begin{equation}
\left. \frac{d}{dt}\right\vert _{t=0}I\left( M,\phi _{t}^{\ast }\left( \rho
^{-4}g\right) ,\chi \right) =0.  \label{eq5.6}
\end{equation}%
On the other hand,%
\begin{eqnarray}
I\left( M,\phi _{t}^{\ast }\left( \rho ^{-4}g\right) ,\chi \right)
&=&I\left( M,\rho ^{-4}g,\chi \circ \phi _{-t}\right)  \label{eq5.7} \\
&=&-2\left( \int_{M}\chi \circ \phi _{-t}\rho ^{-6}d\mu \right) ^{\frac{1}{3}%
}\int_{M}P\rho \cdot \rho d\mu .  \notag
\end{eqnarray}%
Since%
\begin{equation}
\int_{M}P\rho \cdot \rho d\mu =-\int_{M}\chi \rho ^{-6}d\mu <0,
\label{eq5.8}
\end{equation}%
we see $\int_{M}X\chi \cdot \rho ^{-6}d\mu =0.$
\end{proof}

\begin{corollary}
\label{cor5.1}Assume $\chi $ and $\rho $ are positive smooth functions on $%
S^{3}$ such that%
\begin{equation}
P\rho =-\chi \rho ^{-7},  \label{eq5.9}
\end{equation}%
then%
\begin{equation}
\int_{S^{3}}\left\langle \nabla \chi \left( x\right) ,\nabla
x_{i}\right\rangle \rho \left( x\right) ^{-6}d\mu \left( x\right) =0
\label{eq5.10}
\end{equation}%
for $i=1,2,3,4$.
\end{corollary}

We are ready to give a perturbation proof of Theorem \ref{thm1.1}.

\begin{proof}[Proof of Theorem \protect\ref{thm1.1}]
For $\varepsilon >0$ small let $u$ be a minimizer of the perturbation
problem (\ref{eq2.1}) with $\left\Vert u^{-1}\right\Vert _{L^{6}}=1$, then%
\begin{equation}
Pu+\varepsilon u=-s_{\varepsilon }u^{-7}.  \label{eq5.11}
\end{equation}%
In another word,%
\begin{equation}
Pu=-\left( s_{\varepsilon }+\varepsilon u^{8}\right) u^{-7}.  \label{eq5.12}
\end{equation}%
It follows from Section \ref{sec4} that $u$ must be radial symmetric and
decreasing with respect to some point on $S^{3}$. By rotation we can assume
that point is the north pole $N$. By Corollary \ref{cor5.1} we have%
\begin{equation}
\int_{S^{3}}\left\langle \nabla \left( s_{\varepsilon }+\varepsilon
u^{8}\right) ,\nabla x_{4}\right\rangle u\left( x\right) ^{-6}d\mu \left(
x\right) =0.  \label{eq5.13}
\end{equation}%
In another word,%
\begin{equation}
\int_{S^{3}}\left\langle \nabla u,\nabla x_{4}\right\rangle u\left( x\right)
d\mu \left( x\right) =0.  \label{eq5.14}
\end{equation}%
Since $\left\langle \nabla u,\nabla x_{4}\right\rangle \geq 0$, we see $%
\left\langle \nabla u,\nabla x_{4}\right\rangle =0$ and $u$ must be constant
function. It follows that%
\begin{equation}
s_{\varepsilon }=\left( \frac{15}{16}-\varepsilon \right) \left\vert
S^{3}\right\vert ^{\frac{4}{3}}.  \label{eq5.15}
\end{equation}%
Hence for any $\varphi \in H^{2}\left( S^{3}\right) $, $\varphi >0$ and $%
\varepsilon >0$ small,%
\begin{equation}
\left\Vert \varphi ^{-1}\right\Vert _{L^{6}}^{2}\left( E\left( \varphi
\right) +\varepsilon \left\Vert \varphi \right\Vert _{L^{2}}^{2}\right) \geq
-\left( \frac{15}{16}-\varepsilon \right) \left\vert S^{3}\right\vert ^{%
\frac{4}{3}}.  \label{eq5.16}
\end{equation}%
Let $\varepsilon \downarrow 0$, we get%
\begin{equation}
\left\Vert \varphi ^{-1}\right\Vert _{L^{6}}^{2}\left( E\left( \varphi
\right) +\varepsilon \left\Vert \varphi \right\Vert _{L^{2}}^{2}\right) \geq
-\frac{15}{16}\left\vert S^{3}\right\vert ^{\frac{4}{3}}.  \label{eq5.17}
\end{equation}
\end{proof}

At last we want to point out that the study of extremal problem (\ref{eq1.14}%
) is motivated by the question of finding conformal metrics with constant $Q$
curvature (see \cite{XY}). One of the crucial ingredient in our approach to (%
\ref{eq1.10}) is the condition NN. The validity of condition NN on the
standard $S^{3}$ is in some sense straightforward and had been observed in 
\cite{HY1}. For general metrics, the understanding of condition NN is more
recent and depends heavily on an identity found in \cite[section 2]{HY2}.
More precisely, if $\left( M,g\right) $ is a smooth compact Riemannian
manifold with $Y\left( g\right) >0$ and $Q>0$, then the following three
statements are equivalent (see \cite[section 4]{HY6}):

\begin{itemize}
\item Extremal problem (\ref{eq1.14}) is achieved;

\item The second eigenvalue of Paneitz operator $\lambda _{2}\left( P\right)
>0$;

\item The Paneitz operator $P$ satisfies the condition NN.
\end{itemize}

On the other hand, one can solve the constant $Q$ curvature problem without
the solution to extremal problem (\ref{eq1.14}) (see \cite{HY3,HY5}). It is
still not known whether we can find a smooth compact Riemannian manifold $%
\left( M^{3},g\right) $ with $Y\left( g\right) >0$, $Q>0$ and the Paneitz
operator has two or more negative eigenvalues. For such kind of metrics, the
value of (\ref{eq1.14}) would be $-\infty $ and the Paneitz operator does
not satisfy condition NN.

\end{document}